\documentclass{article}

\usepackage[english]{babel}
\usepackage{amssymb}
\usepackage{amsmath}
\usepackage{amsthm}
\usepackage{color}
\usepackage{mathrsfs}

\newtheorem{theorem}{Theorem}[section]
\newtheorem{definition}[theorem]{Definition}
\newtheorem{lemma}[theorem]{Lemma}
\newtheorem{proposition}[theorem]{Proposition}
\newtheorem{corollary}[theorem]{Corollary}
\newtheorem{remark}[theorem]{Remark}

\newtheorem{problem}{Problem}

\newcommand{\hh}{{\mathbb{H}}}

\newcommand{\cc}{{\mathbb{C}}}
\newcommand{\rr}{{\mathbb{R}}}
\newcommand{\nn}{{\mathbb{N}}}
\newcommand{\s}{{\mathbb{S}}}
\newcommand{\z}{{\mathcal{Z}}}

\newcommand{\mr}{{\mathcal{M}}}

\newcommand{\B}{{\mathbb{B}}}

\begin{document}

\markboth{Cinzia Bisi, Caterina Stoppato}
{Landau's theorem for slice regular functions on the quaternionic unit ball}


\date{}

\title{\bf Landau's theorem for slice regular functions on the quaternionic unit ball}

\author{Cinzia Bisi
\\
\small Universit\`a degli Studi di Ferrara\\
\small Dipartimento di Matematica e Informatica\\
\small Via Machiavelli 35, I-44121 Ferrara, Italy\\
\small cinzia.bisi@unife.it\\
\and
Caterina Stoppato$^*$
\\
\small Universit\`a degli Studi di Firenze\\
\small Dipartimento di Matematica e Informatica ``U. Dini''\\
\small Viale Morgagni 67/A, I-50134 Firenze, Italy\\
\small stoppato@math.unifi.it
}


\maketitle

\begin{abstract}
Along with the development of the theory of \emph{slice regular} functions over the real algebra of quaternions $\hh$ during the last decade, some natural questions arose about slice regular functions on the open unit ball $\B$ in $\hh$. This work establishes several new results in this context. Along with some useful estimates for slice regular self-maps of $\B$ fixing the origin, it establishes two variants of the quaternionic Schwarz--Pick lemma, specialized to maps $\B\to\B$ that are not injective. These results allow a full generalization to quaternions of two theorems proven by Landau for holomorphic self-maps $f$ of the complex unit disk with $f(0)=0$. Landau had computed, in terms of $a:=|f'(0)|$, a radius $\rho$ such that $f$ is injective at least in the disk $\Delta(0,\rho)$ and such that the inclusion $f(\Delta(0,\rho))\supseteq\Delta(0,\rho^2)$ holds. The analogous result proven here for slice regular functions $\B\to\B$ allows a new approach to the study of Bloch--Landau-type properties of slice regular functions $\B\to\hh$.
\end{abstract}


\section{Introduction}

The unit ball in the real algebra of quaternions $\hh$, namely
\[\B:=\{q \in \hh : |q|<1\}\,,\]
is the subject of intensive investigation within the theory of slice regular quaternionic functions introduced in~\cite{cras,advances}. The theory is based on the next definition, where the notation $\s := \{ q \in \hh\,\, | \,\, q^{2} = -1 \}$ is used for the sphere of quaternionic \emph{imaginary units}.

\begin{definition}\label{def:sliceregular}
Let $\Omega$ be a domain (an open connected set) in $\hh$ and let $f \colon \Omega \to \hh.$ For all $I \in \s,$ let us use the notations $L_I:=\rr+I\rr$, $\Omega_I:= \Omega\,\cap\,L_I$ and $f_I:=f_{|_{\Omega_I}}.$ The function $f$ is called (Cullen or) \emph{slice regular} if, for all $I \in \s$, the restriction $f_{I}$ is holomorphic; that is, if, for all $I \in \s$, $f_I$ is differentiable and the function $\overline{\partial}_I f \colon \Omega_I \to \hh$ defined by 
\[\overline{\partial}_I f (x+yI) : = \frac{1}{2} \left(\frac{\partial}{\partial x} + I \frac{\partial}{\partial y}\right) f_I (x+Iy)\]
vanishes identically. If this is the case, then a slice regular function $\partial_c f : \Omega \to \hh$ can be defined by setting
\[\partial_c f (x+yI) : = \frac{1}{2} \left(\frac{\partial}{\partial x} -I \frac{\partial}{\partial y}\right) f_I(x+yI)\]
for $I \in \s,\ x,y \in \rr$ (such that $x+Iy \in \Omega$). It is called the \emph{Cullen derivative} of $f$.
\end{definition}

For slice regular functions on the quaternionic unit ball $\B$, the Schwarz lemma and its boundary version were proven in~\cite{advances,rigidity}. Slice regular analogs of the M\"obius transformations of $\B$ have been introduced and studied in~\cite{volumeindam,fixedmoebius,moebius}, leading to the generalization of the Schwarz--Pick lemma in~\cite{BSIndiana}. Other results concerning slice regular functions on $\B$ have been published in \cite{alpaybolotnikovcolombosabadini,invariantmetrics,bohrregular,blochlandau,galgonzalezsabadini2,galgonzalezsabadini,landautoeplitz}.

Within this rich panorama, the present work establishes the quaternionic counterparts of the following results due to Landau \cite{landau1929,librolandau}, which we quote in the form of~\cite{libroheins}, \S 2.10.

\begin{theorem}[Landau]\label{landau1}
For each $a \in (0,1)$, let $\Phi_a$ denote the set of holomorphic self-maps $f$ of the complex unit disk $\Delta(0,1)$ such that $f(0)=0, |f'(0)|=a$. Let
\begin{align*}
&\rho:=\inf_{f \in \Phi_a} r(f)\,,\\
&r(f):= \sup\{r \in (0,1)\, |\, f\mathrm{\ is\ injective\ in\ }\Delta(0,r)\}	\,.
\end{align*}
Then $\rho=\frac{1-\sqrt{1-a^2}}{a}$. Furthermore, for $f\in \Phi_a$ the equality $r(f)=\rho$ holds if, and only if, there exists $\eta \in \partial\Delta(0,1)$ such that $f(z) = F(\eta z)\eta^{-1}$ where
\[F(z):=z\frac{a-z}{1-az}\,.\]
\end{theorem}

\begin{theorem}[Landau]\label{landau2}
For each $a \in (0,1)$, let
\begin{align*}
&P:=\inf_{f \in \Phi_a} R(f)\,,\\
&R(f):= \sup\{r\, |\, \exists\, \Omega_r\subseteq\Delta(0,1)\mathrm{\ s.t.\ }0\in\Omega_r\mathrm{\ and\ } f:\Omega_r\to\Delta(0,r)\mathrm{\ is\ bijective}\}\,.
\end{align*}
Then $P=\rho^2$ with $\rho=\frac{1-\sqrt{1-a^2}}{a}$. Furthermore, for $f\in \Phi_a$ it holds $R(f)=\rho^2$ if, and only if, there exists $\eta \in \partial\Delta(0,1)$ such that $f(z) = F(\eta z)\eta^{-1}$.
\end{theorem}

Besides their independent interest, Theorems~\ref{landau1} and~\ref{landau2} can be used to prove one of the most celebrated results in complex function theory:

\begin{theorem}[Bloch--Landau]\label{complexblochlandau}
Let $f$ be a holomorphic function on a region containing the closure of $\Delta(0,1)$ and suppose $f(0)=0, f'(0)=1$. Then there is a disk $S\subseteq\Delta(0,1)$ on which $f$ is injective and such that $f(S)$ contains a disk of radius $b > 1/72$.
\end{theorem}

The reference \cite{libroconway} presents in Ch. XII \S1 a proof of the Bloch--Landau theorem based on reducing to bounded functions and applying to them a variant of Theorem~\ref{landau2}. The largest value of $b$ for which Theorem~\ref{complexblochlandau} holds is known as \emph{Bloch's constant}. As it is well-known, determining this constant is still a challenging problem nowadays.

While the situation is considerably different in the case of several complex variables~\cite{LandauSCV,chengauthier,durenrudin,harris}, variants of Theorem~\ref{complexblochlandau} hold in the theory of slice regular functions, see~\cite{blochlandau}, and in other hypercomplex generalizations of the theory of one complex variable: see~\cite{rochon} for the class of \emph{T-holomorphic} functions over the bicomplex numbers; and~\cite{guerlebeckmorais2012} for square integrable \emph{monogenic} functions over the reduced quaternions. As we already mentioned, in the present work we establish perfect analogs of Theorems~\ref{landau1} and~\ref{landau2} for slice regular functions. Other original results are proven along with them and a new version of Theorem~6 of~\cite{blochlandau} is obtained as an application. The paper is structured as follows.

In Section~\ref{sec:preliminaries}, we recall some preliminary material needed for our study, including the algebraic structure of slice regular functions: a ring structure with the usual addition $+$ and a multiplication $*$, as well as the existence of a multiplicative inverse $f^{-*}$ when $f\not\equiv0$. This structure is the basis for the construction of \emph{regular M\"obius transformations}, namely, $\mr_{q_0}(q) u$ for $q_0 \in \B, u \in \partial \B$, where 
\[\mr_{q_0}(q):=(1-q\bar q_0)^{-*}*(q-q_0)\,.\]
We recall some known results on the differential of a slice regular function and we derive a characterization of functions that are not injective. The quaternionic Schwarz--Pick lemma is also recalled in detail.

Section~\ref{sec:schwarzpick} establishes two variants of the quaternionic Schwarz--Pick lemma, specialized to self-maps of the quaternionic unit ball that are not injective. The first one is:

\begin{theorem}
Let $f : \B \to \B$ be a regular function, let $q_0 \in \B$ and set $v:=f(q_0)$. If $f$ is not injective in any neighborhood of $f^{-1}(v)$ then there exists $p_0 \in \B$ such that
\[|(f(q) -v) * (1 - \bar v * f(q))^{-*}| \leq |\mr_{q_0} (q) * \mr_{p_0} (q)|.\]
\end{theorem}

A second variant can be derived from the former:

\begin{theorem}
Let $f \colon \B \to \B$ be a slice regular function, which, for some $r \in (0,1)$, is injective in ${B}(0,r):=\{q \in \hh: |q|<r\}$ but is not injective in ${B}(0,r')$ for any $r'>r$.
Then there exists $q_0 \in \partial B(0,r)$ such that $f$ is not injective in any neighborhood of $f^{-1}(f(q_0))$ and
\[
|(f(q) -f(q_0)) * (1 - \overline{f(q_0)}* f(q))^{-*}| \leq |\mr_{q_0}(q) * \mr_{p_0} (q)|.
\]
for some $p_0 \in \overline{B(0,r)}$. In particular, if $f(0) = 0$ then $|f(q_0)| \leq r^2$.
\end{theorem}

We add, in Section~\ref{sec:bounds}, some useful estimates for slice regular self-maps of the unit ball fixing the origin:

\begin{theorem}
Let $f : \B \to \B$ be a slice regular function with $f(0)=0$. If $a := |\partial_cf(0)|$ belongs to $(0,1)$ then
\begin{equation}
|q|\frac{a-|q|}{1-a|q|} \leq |f(q)| \leq |q|\frac{|q|+a}{1+a|q|}.
\end{equation}
for all $q \in \B$. Furthermore, if there exists $q \in \B$ such that equality holds on the left-hand side or on the right-hand side, then $f(q) = q \mr(q)$ where $\mr$ is a regular M\"obius transformation of $\B$ with $|\mr(0)| = a$.
\end{theorem}

The aforementioned theorems allow us to achieve, in Section~\ref{sec:landau}, a full generalization of Landau's results: 

\begin{theorem}
Let $f : \B \to \B$ be a slice regular function with $f(0)=0$. If $a := |\partial_cf(0)|$ belongs to $(0,1)$ and if we set $\rho := \frac{1-\sqrt{1-a^2}}{a}$ then the following properties hold.
\begin{enumerate}
\item The function $f$ is injective at least in the ball $B(0, \rho)$. 
\item For all $r \in (0, \rho)$, $B\big(0, r\frac{a-r}{1-ar}\big) \subseteq  f({B}(0,r)) \subseteq {B}\big(0, r \frac{a+r}{1+ar}\big)$. As a consequence,
\[{B}(0, \rho^2) \subseteq f(B(0, \rho) ).\]
\item The following are equivalent:
\begin{enumerate}
\item $B(0,\rho)$ is the largest ball centered at $0$ where $f$ is injective;
\item there exists a point $q_0 \in \partial {B}(0,\rho)$ with $f(q_0) \in \partial{B}(0, \rho^2)$;
\item $f(q) = q \mr(q)$ where $\mr$ is a regular M\"obius transformation of $\B$ (necessarily such that $\mr(0)=\partial_cf(0)$, whence $|\mr(0)|=a$).
\end{enumerate}
\end{enumerate}
\end{theorem}

In Section~\ref{sec:bloch}, as an application, we obtain a quaternionic Bloch--Landau-type result in the spirit of~\cite{blochlandau}. Although finding a full-fledged quaternionic generalization of Theorem~\ref{complexblochlandau} is still an open problem, the new approach used here opens a new path towards such a generalization, which will be the subject of future research.


\section{Preliminary material}\label{sec:preliminaries}

Let us recall some basic material on slice regular functions, see Chapter~1 in~\cite{librospringer}. 
We will henceforth use the adjective \emph{regular}, for short, to refer to slice regular functions.

\begin{proposition}
The regular functions on a Euclidean ball 
\[B(0,R):=\{q \in \hh\ |\ |q|<R\}\]
are exactly the sums of those power series $\sum_{n \in \nn} q^na_n$ (with $\{a_n\}_{n \in \nn}\subset\hh$) which converge in $B(0,R)$.
\end{proposition}

Between two regular functions on $B(0,R)$, say
\[f(q):=\sum_{n \in \nn} q^n a_n\,,\quad g(q):=\sum_{n \in \nn} q^n b_n\,,\]
the \emph{regular product} is defined as follows:
\[(f*g)(q):= \sum_{n \in \nn} q^n \sum_{k=0}^na_kb_{n-k}\,.\]
The \emph{regular conjugate} of $f$ is defined as
\[f^c (q) :=\sum_{n \in \nn} q^n \bar a_n.\]

The next definition presents a larger class of domains that are of interest in the theory of regular functions.

\begin{definition}
A domain $\Omega \subseteq \hh$ is called a \emph{slice domain} if $\Omega_I= \Omega \cap L_I$ is an open connected subset of $L_I \cong \cc$ for all $I \in \s,$ and $\Omega$ intersects the real axis. A slice domain $\Omega$ is termed \emph{symmetric} if it is symmetric with respect to the real axis, i.e., if for all $x,y\in\rr, I\in\s$ the inclusion $x+Iy \in \Omega_I$ implies $x+y\s\subset\Omega$.
\end{definition}

The definition of the multiplication $(f,g)\mapsto f*g$ can be extended to the case of regular functions on a symmetric slice domain $\Omega$, leading to the next result. For more details, see \S1.4 in~\cite{librospringer}.

\begin{proposition}
Let $\Omega$ be a symmetric slice domain. The set of regular functions on $\Omega$ is a (noncommutative) ring with respect to $+$ and $*$.
\end{proposition}

The operation $f \mapsto f^c$ can also be consistently extended to the case of regular functions on a symmetric slice domain $\Omega$. The additional operation of \emph{symmetrization}, defined by the formula
\[f^s:= f*f^c=f^c*f\]
allows to define the \emph{regular reciprocal} of $f$ as
\[f^{-*}(q):=f^s(q)^{-1} f^c(q)\]
and to prove the next results, where we use the notation $\z_h := \{ q \in \Omega \,\, | \,\,  h(q)=0 \}$ for the zero set of a regular $h:\Omega \to \hh$. For details, we refer the reader to Chapter~5 in~\cite{librospringer}.

\begin{theorem}
Let $f$ be a regular function on a symmetric slice domain $\Omega$ and suppose that $f \not \equiv 0$. Then $f^{-*}$ is regular in $\Omega \setminus \z_{f^s}$, which is a symmetric slice domain, and
\[f*f^{-*} = f^{-*}*f \equiv 1.\]
\end{theorem}

\begin{theorem}\label{formulae}
Let $f,g$ be regular functions on a symmetric slice domain $\Omega$. Then
\[(f*g) (q) = \left\{
		\begin{array}{ll}
                        0  &  \textrm{if} \,\,\,\, f(q)=0, \\
                      f(q)\,g(f(q)^{-1} q f(q)) & \textrm{otherwise}.
                  \end{array} \right.
\]
If we set $T_f (q) := f^c (q)^{-1} q f^c(q)$ for all $q \in \Omega \setminus \mathcal{Z}_{f^s},$ then
\begin{equation}\label{quozreg}
(f^{-*}*g)(q) = f(T_f (q))^{-1} g (T_f (q)),
\end{equation}
for all $q \in \Omega \setminus \mathcal{Z}_{f^s}.$ For all $x,y \in \mathbb{R}$ with $x+y \s \subset \Omega \setminus \mathcal{Z}_{f^s} \subset \Omega \setminus \mathcal{Z}_{f^c},$ the function $T_f$ maps $x+y \s$ to itself (in particular,
 $T_f (x)=x$ for all $x \in \mathbb{R}$). Furthermore, $T_f$ is a diffeomorphism from $\Omega \setminus \mathcal{Z}_{f^s}$ onto itself, with inverse $T_{f^c}.$
\end{theorem}

Let us now recall some material on the zeros of regular quaternionic functions, see Chapter~3 in~\cite{librospringer}. We begin with a result that is folklore in the theory of quaternionic polynomials. For all $q_0 = x_0+ I y_0$ with $x_0,y_0 \in \rr, I \in \s$, we will use the notation
\[S_{q_0}:=x_0 + y_0\s.\]

\begin{theorem}
Let $q_0,q_1\in \hh$ and $f(q)=(q-q_0)*(q-q_1)$. If $q_1 \not \in S_{q_0}$ then $f$ has two zeros in $\hh$, namely $q_0$ and
$(q_0-\bar q_1)q_1(q_0-\bar q_1)^{-1} \in S_{q_1}$. Now suppose, instead, that $q_1 \in S_{q_0}$. If $q_1\neq \bar q_0$ then $f$ only vanishes at $q_0$, while if $q_1=\bar q_0$ then the zero set of $f$ is $S_{q_0}$.
\end{theorem}

As in the case of a holomorphic complex function, the zeros of a regular quaternionic function can be factored out. The relation between the factorization and the zero set is, however, subtler than in the complex case because of the previous theorem.

\begin{theorem}\label{factorization}
Let $f\not\equiv0$ be a regular function on a symmetric slice domain $\Omega$ and let $x+y\s\subset\Omega$. There exist $m\in \nn = \{0, 1, 2, \ldots\}$ and a regular function $g: \Omega \to \hh$, not identically zero in $x+y\s$, such that
\[f (q) = [(q-x)^2+y^2]^m g(q)\]
If $g$ has a zero $p_1\in x+y\s$, then such a zero is unique and there exist $n \in \nn, p_1, \ldots, p_n \in x+y\s$ (with $p_l \neq \bar p_{l+1}$ for all $l \in \{1,\ldots,n-1\}$) such that
\[g(q) = (q-p_1)*\ldots*(q-p_n)*h(q)\]
for some regular function $h: \Omega \to \hh$ that does not have zeros in $x+y\s$.
\end{theorem}

This motivates the next definition.

\begin{definition}
In the situation of Theorem~\ref{factorization}, $f$ is said to have \emph{spherical multiplicity} $2m$ at $x+y\s$ and \emph{isolated multiplicity} $n$ at $p_1$. Finally, the \emph{total multiplicity} of $x+y\s$ for $f$ is defined as the sum $2m+n$.
\end{definition}

Let us now recall a definition originally given in~\cite{perotti} and a few results from~\cite{ocs}, which concern the real differential of a regular function.

\begin{definition}
Let $f$ be a regular function on a symmetric slice domain $\Omega$, and let $q_0=x_0+Iy_0 \in \Omega$ with $x_0,y_0 \in \rr, I \in \s$. If $y_0 \neq 0$ then the \emph{spherical derivative} of $f$ at $q_0$ is defined by the formula
\[\partial_s f(q_0):= (q_0-\bar q_0)^{-1}(f(q_0)-f(\bar q_0))\,.\]
\end{definition}

\begin{proposition}\label{realdifferential}
Let $f$ be a regular function on a symmetric slice domain $\Omega$, and let $q_0=x_0+Iy_0 \in \Omega$ with $x_0,y_0 \in \rr, I \in \s$. If $y_0=0$ then the real differential of $f$ at $q_0$ acts by right multiplication by the Cullen derivative $\partial_c f(q_0)$ on the entire tangent space $T_{q_0}\Omega \simeq \hh$. If, on the other hand, $y_0 \neq 0$ and if we split such space as $\hh = L_I \oplus L_I^\perp$ then the real differential acts on $L_I$ by right multiplication by $\partial_c f(q_0)$ and on $L_I^\perp$ by right multiplication by the spherical derivative $\partial_s f(q_0)$.
\end{proposition}

The next result characterizes the \emph{singular set} $N_f$ of a regular function $f$, that is, the set of points where the real differential of $f$ is not invertible.

\begin{proposition}\label{noninvertibledifferential}
Let $f$ be a regular function on a symmetric slice domain $\Omega$, and let $q_0 \in \Omega$. The real differential of $f$ at $q_0$ is not invertible if, and only if, there exist $\widetilde q_0 \in S_{q_0}$ and a regular function $g:\Omega \to \hh$ such that
\begin{equation} \label{diffnotinv}
f(q) = f(q_0) + (q-q_0)*(q-\widetilde q_0)*g(q);
\end{equation}
that is, if, and only if, $f-f(q_0)$ has total multiplicity $n\geq 2$ at $S_{q_0}$. 
We can distinguish the following special cases:
\begin{itemize}
\item equality~\eqref{diffnotinv} holds with $\widetilde q_0 = \bar q_0$ if, and only if, the spherical derivative $\partial_sf(q_0)$ vanishes;
\item equality~\eqref{diffnotinv} holds with $\widetilde q_0 = q_0$ if, and only if, the Cullen derivative $\partial_cf(q_0)$ vanishes.
\end{itemize}
\end{proposition}

The following theorem asserts that the total multiplicity $n$ of $f-f(q_0)$ at $S_{q_0}$ is constant in a neighborhood of $S_{q_0}$.

\begin{theorem}\label{singular}
Let $\Omega$ be a symmetric slice domain and let $f:\Omega \to \hh$ be a non-constant regular function. Then its singular set $N_f$ has empty interior. Moreover, for a fixed $q_0 \in N_f$, let $n > 1$ be the total multiplicity of $f-f(q_0)$ at $S_{q_0}$. Then there exist a neighborhood $U$ of $q_0$ and a neighborhood $T$ of $S_{q_0}$ such that, for all $q_1 \in U$, the sum of the total multiplicities of the zeros of $f-f(q_1)$ in $T$ equals $n$; in particular, for all $q_1 \in U \setminus N_f$ the preimage of $f(q_1)$ includes at least two distinct points of $T$.
\end{theorem}

\begin{corollary}\label{diffeo}
Let $\Omega$ be a symmetric slice domain and let $f: \Omega \to \hh$ be a regular function. If $f$ is injective, then its singular set $N_f$ is empty.
\end{corollary}

For the purposes of our present work, we add the following remark.

\begin{proposition}\label{violationofinjectivity}
Let $\Omega$ be a symmetric slice domain and let $f: \Omega \to \hh$ be a regular function. For each value $v$ of $f$, the following are equivalent:
\begin{itemize}
\item $f$ is not injective in any neighborhood of $f^{-1}(v);$
\item there exist $q_0,q_1 \in \Omega$ and a regular $g: \Omega \to \hh$ such that 
\begin{equation}\label{factorize}
f(q) = v + (q-q_0)*(q-q_1)*g(g).
\end{equation}
\end{itemize}
\end{proposition}

\begin{proof}
Suppose \[f(q) = v + (q-q_0)*(q-q_1)*g(g),\]
so that in particular $f(q_0)=v$.
If $q_1 \not \in S_{q_0}$ then $f-v$ vanishes at some $\widetilde q_1 \in S_{q_1}$. As a consequence, $f(\widetilde q_1)=v=f(q_0)$ and $f$ is not injective on any neighborhood of $f^{-1}(v) \supseteq \{q_0,\widetilde q_1\}$. If, on the other hand, $q_1 \in S_{q_0}$ then $q_0 \in N_f$ by Proposition~\ref{noninvertibledifferential}. In such a case, $f$ is not injective in any neighborhood of $f^{-1}(v)\ni q_0$ by Theorem~\ref{singular}.

Now let us prove the converse implication. If $v$ is the value of $f$ at $q_0$ then there exists $h: \Omega \to \hh$ such that $f(q)=v+(q-q_0)*h(q)$. If $h$ does not admit any zero $q_1 \in \Omega$ then $f$ does not take the value $v$ at any other point of $\Omega$ other than $q_0$ and $q_0 \not \in N_f$. As a consequence, $f^{-1}(v) = \{q_0\}$ and $f$ is a local diffeomorphism near $q_0$. A fortiori, $f$ is injective in a neighborhood of $f^{-1}(v)$.
\end{proof}

We conclude our overview of preliminary material with a few results concerning the unit ball $\B:= \{ q \in \hh\,\, | \,\, |q| <1 \}$. The work \cite{moebius} introduced the \emph{regular M\"obius transformations}, namely the functions $q\mapsto\mr_{q_0}(q)u$ with $u \in \partial\B$, $q_0 \in \B$ and
\[\mr_{q_0} (q):= (q-q_0)*(1-\bar q_0*q)^{-*} =(q-q_0)*(1-q\bar q_0)^{-*}= (1-q\bar q_0)^{-*}*(q-q_0).\]
These transformations are the only bijective self-maps of $\B$ that are regular. In \cite{BSIndiana}, the following result has been proven.

\begin{theorem}[Schwarz-Pick lemma]\label{schwarzpick}
Let $f : \B \to \B$ be a regular function and let $q_0 \in \B$. Then in $\B$:
\begin{eqnarray}\label{schwarzpickinequality}
|(f(q)-f(q_0))*(1-\overline{f(q_0)}*f(q))^{-*}| &\leq& |\mr_{q_0}(q)|.
\end{eqnarray}
Moreover,
\begin{eqnarray}
|\partial_c f *(1-\overline{f(q_0)}*f(q))^{-*}|_{|_{q_0}} &\leq& \frac{1}{1-|q_0|^2}\\
\frac{|\partial_s f(q_0)|}{|1-f^s(q_0)|} &\leq& \frac{1}{|1-\overline{q_0}^2|}
\end{eqnarray}
If $f$ is a regular M\"obius transformation of $\B$, then equality holds in the previous formulas. Else, all the aforementioned inequalities are strict (except for the first one at $q_0$, which reduces to $0\leq0$).
\end{theorem}

The proof was based on the following lemmas, proven in~\cite{volumeindam} and in~\cite{BSIndiana}, respectively.

\begin{lemma}\label{introducezero}
If $f : \B \to \B$ is a regular function then for all $q_0 \in \B$, the function $\widetilde f(q) := (f(q)-f(q_0))*(1-\overline{f(q_0)}*f(q))^{-*}$ is a regular function from $\B$ to itself with $\widetilde f(q_0)=0$.
\end{lemma}

\begin{lemma}\label{casewithazero}
If $f : \B \to \B$ is a regular function having a zero at $q_0 \in \B$, then $\mr_{q_0}^{-*}*f$ is a regular function from $\B$ to itself.
\end{lemma}

The proof of Lemma~\ref{casewithazero} used the next result (see \S7.1 in~\cite{librospringer}), which will be thoroughly used in the present work.

\begin{theorem}[Maximum Modulus Principle]\label{MaxPrinc}
Let $\Omega \subseteq \hh$ be a slice domain and let $f \colon \Omega \to \hh$ be regular. If $|f|$ has a relative maximum at $p \in \Omega,$ then $f$ is constant. As a consequence, if $\Omega$ is bounded and if, for all $q_0 \in \partial \Omega$, 
\[\limsup_{\Omega \ni q \to q_0} |f(q)| \leq M\]
then $|f|\leq M$ in $\Omega$ and the inequality is strict unless $f$ is constant.
\end{theorem}

Finally, the following lemma, proven in~\cite{BSIndiana} as a further tool for the proof of Theorem~\ref{schwarzpick}, will also be useful later in this paper.

\begin{lemma}\label{moduli}
Let $f,g,h: B=B(0,R) \to \hh$ be regular functions. If $|f|\leq |g|$ then $|h*f| \leq |h*g|$. Moreover, if $|f|< |g|$ then $|h*f| < |h*g|$ in $B \setminus \z_h$, where we recall that $\z_h := \{ q \in B \,\, | \,\,  h(q)=0 \}$.
\end{lemma}


\section{A generalized Schwarz-Pick lemma}\label{sec:schwarzpick}

Our first step towards a quaternionic version of Landau's results is a special variant of Theorem~\ref{schwarzpick}, concerned with self-maps of $\B$ that are not injective. We will start with the next theorem and then achieve the result which we will apply later in the paper.

\begin{theorem}\label{schwarzpickvariant}
Let $f : \B \to \B$ be a regular function, let $q_0 \in \B$ and set $v:=f(q_0)$. If $f$ is not injective in any neighborhood of $f^{-1}(v)$ then there exists $p_0 \in \B$ such that
\[|(f(q) -v) * (1 - \bar v * f(q))^{-*}| \leq |\mr_{q_0} (q) * \mr_{p_0} (q)|.\]
Namely, if $q_0,q_1 \in \B$ are such that $f(q) = v + (q-q_0)*(q-q_1)*g(g)$ holds for some regular $g:\B\to \hh$, then the previous inequality holds with $p_0:= (1-q_1q_0)^{-1} q_1 (1-q_1q_0)\in S_{q_1}$.
\end{theorem}

\begin{proof}
If $f$ is not injective in any neighborhood of $f^{-1}(v)$ then Proposition~\ref{violationofinjectivity} applies. Let $q_0,q_1 \in \B$ be such that 
\[f(q) = v + (q-q_0)*(q-q_1)*g(g)\]
for some regular $g: \B \to \hh$. Now,
\begin{align*}
\widetilde f(q) :=& (f(q) -v) * (1 - \bar v * f(q))^{-*}\\
=& (q-q_0)*(q-q_1)*g(g)*(1 - \bar v * f(q))^{-*}
\end{align*}
is a self-map of $\B$ with a zero at $q_0$.
In particular, by Lemma~\ref{casewithazero}, $h := \mr_{q_0}^{-*}*\widetilde f$ is a regular function from $\B$ to itself.
Moreover,
\begin{align*}
h(q)&= (1-q\bar q_0)*(q-q_1)*g(g)*(1 - \bar v * f(q))^{-*}
\end{align*}
vanishes at $p_0:= (1-q_1q_0)^{-1} q_1 (1-q_1q_0)\in S_{q_1}$ by Theorem~\ref{formulae}. By applying again Lemma~\ref{casewithazero} we find that
\[\mr_{p_0}^{-*}*h = \mr_{p_0}^{-*}*\mr_{q_0}^{-*}*\widetilde f\]
is again a regular function from $\B$ to itself.
By Lemma~\ref{moduli} we conclude that
\[|\widetilde f|\leq |\mr_{q_0}*\mr_{p_0}|\]
in $\B$, as desired.
\end{proof}

The previous result was already known in the special cases when the violation of injectivity is caused by the vanishing of the Cullen or the spherical derivative at $q_0$: see Theorems 5.2 and 5.4 in~\cite{BSIndiana}. In such cases, the point $p_0$ appearing in the statement coincides with $q_0$ or $\bar q_0$, respectively.

We now exploit Theorem~\ref{schwarzpickvariant} and turn it into a result that will be particularly useful in the sequel.

\begin{theorem}\label{globaltolocal}
Let $f \colon \B \to \B$ be a regular function, which, for some $r \in (0,1)$, is injective in ${B}(0,r)$ but is not injective in ${B}(0,r')$ for any $r'>r$.
Then there exists $q_0 \in \partial B(0,r)$ such that $f$ is not injective in any neighborhood of $f^{-1}(f(q_0))$ and
\[
|(f(q) -f(q_0)) * (1 - \overline{f(q_0)}* f(q))^{-*}| \leq |\mr_{q_0}(q) * \mr_{p_0} (q)|.
\]
for some $p_0 \in \overline{B(0,r)}$. In particular, if $f(0) = 0$ then $|f(q_0)| \leq |q_0||p_0|\leq r^2$.
\end{theorem}

\begin{proof}
For each $n\geq1$, since $f$ is not injective in ${B}(0,r+1/n)$, there exist two distinct points $p_n,q_n$ in ${B}(0,r+1/n)$ where $f$ takes the same value $v_n$. Since we supposed $f$ to be injective in $B(0,r)$, only one of the two points, say $p_n$, may be included in $B(0,r)$ while $q_n$ must have $r \leq |q_n| < r+ 1/n$. Therefore, $|q_n| \to r$ as $n\to+\infty$ and, up to refining the sequence, $q_n \to q_0$ for some $q_0 \in \partial B(0,r)$. Up to further refinements, $p_n \to p$ for some $p \in \overline{B(0,r)}$ and $v_n \to v \in \B$. Clearly, $f(q_0)=v=f(p)$. This immediately proves that $f$ is not injective near $f^{-1}(v) \supset \{q_0,p\}$, unless $q_0=p$. In this last case, we can still prove that $f$ is not injective in any neighborhood $U$ of $q_0$, as follows: by construction, $U$ includes two distinct points $p_n,q_n$ (for some $n \in \nn$) where $f(p_n) = v_n = f(q_n)$. 

We have thus proven that $f$ is not injective in any neighborhood of $f^{-1}(v) \supseteq\{q_0,p\}$. If we let $q_1$ be a point of $S_p$ such that \eqref{factorize} holds and if we set $p_0 := (1-q_1q_0)^{-1} q_1 (1-q_1q_0)\in S_p \subset \overline{B(0,r)}$, then by Theorem~\ref{schwarzpickvariant},
\[
|(f(q) -v) * (1 - \bar v * f(q))^{-*}| \leq |\mr_{q_0} * \mr_{p_0} (q)|.
\]
Under the additional hypothesis that $f(0)=0$, we now compute both hands of the inequality at $q=0$. For the left-hand side, we use the fact that $f(0) = f^c(0) = f^s(0) = 0$.
If we set $g(q) = f(q) -v$ then $g(0) = f(0)-v = -v$. If $h(q) =1 - \bar v * f(q)$ then $h^c(q) = 1 - f^c(q) v$, so that $h^c(0) = 1$; and $h^s(q) = (1 - \bar v * f(q))*(1 - f^c(q) v) = 1- \bar v * f(q) - f^c(q) v + \bar v * f^s(q) v$, so that $h^s(0)=1$. Therefore,
\[h^{-*}(0) = h^s(0)^{-1} h^c(0) = 1.\]
Finally,
\[g * h^{-*} (0) = g(0) h^{-*} (0) = - v.\]
As for the right-hand side of the inequality,
\[\mr_{q_0} * \mr_{p_0} (0) = \mr_{q_0} (0)\, \mr_{p_0} (0) = q_0\,p_0.\]
Thus, the inequality implies that
\[|v| \leq |q_0|\, |p_0| \leq r^2,\]
as desired.
\end{proof}


\section{Bounds for regular self-maps of the unit ball fixing the origin}\label{sec:bounds}

In the present section, we will establish upper and lower bounds for regular functions $\B\to\B$ that fix the origin $0$.
In addition to the regular M\"obius transformations
\[\mr_a(q)u = (1-q\bar a)^{-*}*(q-a)*u,\quad a \in \B, u \in \partial\B\,,\]
which we encountered in the previous sections, we will use the \emph{classical quaternionic M\"obius transformations}; namely, $v^{-1}M_a(q)u$ 
with $a \in \B, u,v \in \partial\B$ and with
\[M_a (q):= (1-q\bar a)^{-1}(q-a).\]
The latter transformations are studied in literature and well understood: they are a special case of the class studied in \cite{ahlforsmoebius}; see also the more recent \cite{poincare,poincaretrends}. A comparison between classical and regular M\"obius transformations is undertaken in \cite{volumeindam}.
The following  property, which is well-known in the complex case, will be very useful in the sequel.

\begin{lemma}
Fix $b \in \B$. Then, for all $q \in \B$,
\begin{equation}\label{eq:quatmoebincrease}
 \frac{|b|-|q|}{1-|b||q|} \leq |M_b(q)| \leq \frac{|q|+|b|}{1+|b||q|}.
\end{equation}
Moreover, equality holds in the left-hand side if, and only if, $q=rb$ for some $r\in[0,1]$; and it holds in the right-hand side if, and only if, $q=rb$ for some $r\in(-\infty,0]$.
\end{lemma}
\begin{proof}
By direct computation,
\begin{align*}
&1-|b||q|\leq |1-\bar b q|\leq1+|b||q|\\
&\Rightarrow \frac{1}{(1+|b||q|)^2} \leq \frac{1}{\big|1-\bar b q\big|^2} \leq \frac{1}{(1-|b||q|)^2}\\
& \Rightarrow \frac{1+|b|^2|q|^2-|b|^2-|q|^2}{(1+|b||q|)^2} \leq \frac{1+|b|^2|q|^2-|b|^2-|q|^2}{\big|1-\bar b q\big|^2} \leq \frac{1+|b|^2|q|^2-|b|^2-|q|^2}{(1-|b||q|)^2}\\
&\Rightarrow \frac{\big|1+|b||q|\big|^2-\big||q|+|b|\big|^2}{(1+|b||q|)^2} \leq \frac{\big|1-\bar b q\big|^2-\big|q-b\big|^2}{\big|1-\bar b q\big|^2} \leq \frac{\big|1-|b||q|\big|^2-\big||q|-|b|\big|^2}{(1-|b||q|)^2}\\
&\Rightarrow 1-\frac{\big||q|+|b|\big|^2}{(1+|b||q|)^2} \leq 1-\frac{\big|q-b\big|^2}{\big|1-\bar b q\big|^2} \leq 1-\frac{\big||q|-|b|\big|^2}{(1-|b||q|)^2}\\
&\Rightarrow \frac{\big||q|-|b|\big|}{1-|b||q|} \leq \frac{\big|q-b\big|}{\big|1-\bar b q\big|} \leq \frac{\big||q|+|b|\big|}{1+|b||q|}\\
&\Rightarrow  \frac{|b|-|q|}{1-|b||q|} \leq |M_b(q)| \leq \frac{|q|+|b|}{1+|b||q|}.
\end{align*}
Moreover, equality $|1-\bar b q|= 1\pm|b||q|$ holds if and only if $\bar b q = \mp |b||q|$.
\end{proof}

\begin{remark}\label{quatmoebincrease}
If $a \in (0,1)$ then $M_a$ coincides with the regular transformation $\mr_a$. Such an $\mr_a=M_a$ and its inverse function $\mr_{-a}=M_{-a}$ fix $-1$ and $1$ and they map $(-1,1)$ bijectively into itself in a monotone increasing fashion.
\end{remark}

We now establish the announced upper and lower bounds for regular self-maps of $\B$ that fix the origin.

\begin{theorem}\label{minmax}
Let $f : \B \to \B$ be a regular function with $f(0)=0$. If $a := |\partial_cf(0)|$ belongs to $(0,1)$ then
\begin{equation}\label{eq:minmax}
|q|\frac{a-|q|}{1-a|q|} \leq |f(q)| \leq |q|\frac{|q|+a}{1+a|q|}.
\end{equation}
for all $q \in \B$. Furthermore, if there exists $q \in \B$ such that equality holds on the left-hand side or on the right-hand side, then $f(q) = q \mr(q)$ where $\mr$ is a regular M\"obius transformation of $\B$ with $|\mr(0)| = a$.
\end{theorem}

\begin{proof}
By hypothesis, $f(q) = q*g(q) = q g(q)$ for some $g: \B \to \hh$ with $g(0) = \partial_c f(0)$. Up to rotating both $f$ and $g$, we may suppose that $g(0) =a$. Moreover, 
\[|g(q)| = \frac{|f(q)|}{|q|} \leq \frac{1}{|q|}.\]
For any $r \in (0,1)$, we can conclude by the Maximum Modulus Principle~\ref{MaxPrinc}, that $|g(q)|\leq 1/r$ for all $q \in B(0,r)$. Hence, $g(\B) \subseteq \B$. By Lemma~\ref{introducezero},
\[\widetilde g(q) := (g(q)-a)*(1-a*g(q))^{-*}\]
is a regular self-map of $\B$ with $\widetilde g(0)=0$. Moreover, by Proposition~2 of~\cite{volumeindam}
\[g(q) = (\widetilde g(q)+a)*(1+a*\widetilde g(q))^{-*} = (1+ \widetilde g(q)a)^{-*}* (\widetilde g(q)+a)\,.\]
By formula \eqref{quozreg}, after setting $T(q):= (1+\widetilde g^c(q)a)^{-1}q(1+\widetilde g^c(q)a)$, we have that
\[g = M_{-a} \circ \widetilde g \circ T.\]
If we set $h:=\widetilde g \circ T$ then
\[f(q) = q M_{-a}(h(q))\]
whence, by inequality~\eqref{eq:quatmoebincrease}, 
\[|q|\frac{a-|h(q)|}{1-a|h(q)|} \leq |f(q)| \leq |q| \frac{|h(q)|+a}{1+a|h(q)|}.\]
Moreover, 
\[|h(q)| = |\widetilde g(T(q))| \leq |T(q)| = |q|\]
since (by the Schwarz lemma) $|\widetilde g(p)| \leq |p|$ for all $p \in \B$. For the same reason, $|h(q)| = |q|$ for some $q \in \B$ if, and only if, $\widetilde g$ is a regular rotation, i.e., $\widetilde g (q) = qu$ for some $u \in \partial \B$. If we take into account that $M_{-a}$ is strictly monotone increasing on $(-1,1)$, see Remark \ref{quatmoebincrease}, then
\[\frac{|h(q)|+a}{1+a|h(q)|} \leq \frac{|q|+a}{1+a|q|}\]
and equality holds  if, and only if, $\widetilde g$ is a regular rotation.
Similarly, since $-M_a$ is strictly monotone decreasing on $(-1,1)$ we conclude that
\[\frac{a-|q|}{1-a|q|} \leq \frac{a-|h(q)|}{1-a|h(q)|}\] 
and that an equality holds if, and only if, $\widetilde g$ is a regular rotation.
We have thus proven inequality~\eqref{eq:minmax} and shown that any equality in~\eqref{eq:minmax} implies that $f(q) = q \mr(q)$ where $\mr$ is a regular M\"obius transformation of $\B$. In such a case, since $a := |\partial_cf(0)|$, necessarily $|\mr(0)|=a$.
\end{proof}


\section{A quaternionic version of Landau's results}\label{sec:landau}

The announced extension of Landau's Theorems~\ref{landau1} and~\ref{landau2} to regular quaternionic functions is achieved in the present section. We begin by studying a special class of functions that will play an important role in our main result.

\begin{lemma}\label{extremalcase}
Let $F(q) := q \mr(q)$ where $\mr$ is any regular M\"obius transformation of $\B$. Suppose that $a := |\mr(0)|$ is not zero and set $\rho:=\frac{1-\sqrt{1-a^2}}{a}$. Then there exists a point $q_0$ with $|q_0|=\rho$ where the Cullen derivative $\partial_c F$ vanishes and such that $|F(q_0)|=\rho^2$. In particular, if $\Phi(q) = - q \mr_a(q)$ then
\begin{enumerate}
\item$\partial_c \Phi(\rho)=0$;
\item $\Phi(\rho)=\rho^2$ or, equivalently, $-\mr_a(\rho) = \rho$;
\item $\Phi$ maps $(-1,\rho)$ to $(-1,\rho^2)$ in a monotone increasing fashion and $(\rho,1)$ to $(-1,\rho^2)$ in a monotone decreasing fashion.
\end{enumerate}
\end{lemma}

\begin{proof}
Let us fix $a \in (0,1)$. We start with the special case of the function 
\[\Phi(q) := q* (1-qa)^{-*} *(a-q)= (1-qa)^{-1} q (a-q),\]
for which we make the following remarks.
\begin{enumerate}
\item By direct computation, 
\begin{align*}
\partial_c \Phi(q) &=(1-qa)^{-2} [ (1-qa) (a-q) + qa (a-q) - (1-qa) q]\\
& = (1-qa)^{-2}(aq^2-2q+a),
\end{align*}
whose only zero inside $\B$ is $\rho$. 
\item By the definition of $\Phi$, $\Phi(\rho) = \rho^2$ if, and only if, $-\mr_a(\rho) = \rho$. This is equivalent to $a-\rho = (1-\rho a) \rho$, that is, to $a\rho^2 -2\rho + a = 0$, which is true by the definition of $\rho$.
\item Since $\mr_a$ fixes $-1$ and $1$, the function $\Phi$ maps both points to $-1$. Moreover, we have just proven that $\Phi$ maps $\rho$ to $\rho^2$. The thesis follows by observing that $\Phi$ must be monotone on either interval $(-1,\rho),(\rho,1)$ since the Cullen (hence the real) derivative only vanishes at $\rho$.
\end{enumerate}
Now, a regular M\"obius transformation $\mr_b u$ (with $b \in \B, u \in \partial \B$) maps $0$ to $a$ if, and only if, $a = \mr_b(0) u = -b u$, that is, $b=-a\bar u$. If we set
\[\Phi_u(q) := q \mr_{-a\bar u}(q) u\]
then restricting to the plane $L_I$ that includes $u$ we get that, for all $z \in \B_I = \B \cap L_I$,
\[\Phi_u(z) = z \frac{z+a\bar u}{1+zua} u = z u \bar u \frac{zu+a}{1+zua} = -\Phi(-zu) \bar u\]
and $\partial_c \Phi_u(z) = \partial_c \Phi(-zu)$. As a consequence, for $q_0 := -\rho \bar u$ we can compute $|\Phi_u(q_0)|=\rho^2$ and $\partial_c\Phi_u(q_0) =0$.

Finally, let us consider $F(q):=q \mr(q)$ where $\mr$ is any regular M\"obius transformation of $\B$ such that $|\mr(0)| = a$. If we set 
\[v := \frac{\overline{\mr(0)}}{|\mr(0)|}\]
then $\mr (0) v = a,$ and $\mr v$ is a regular M\"obius transformation of $\B$ mapping $0$ to $a$. Hence, $F(q) v = q \mr(q) v = \Phi_u(q)$ for some $u \in \B$ and the thesis follows from what we have already proven for $\Phi_u$.
\end{proof}

We are now ready for the announced result.

\begin{theorem}\label{landau}
Let $f : \B \to \B$ be a regular function with $f(0)=0$. If $a := |\partial_cf(0)|$ belongs to $(0,1)$ and if we set $\rho := \frac{1-\sqrt{1-a^2}}{a}$ then the following properties hold.
\begin{enumerate}
\item The function $f$ is injective at least in the ball $B(0, \rho)$. 
\item For all $r \in (0, \rho)$, $B\big(0, r\frac{a-r}{1-ar}\big) \subseteq  f({B}(0,r)) \subseteq {B}\big(0, r \frac{a+r}{1+ar}\big)$. As a consequence,
\[{B}(0, \rho^2) \subseteq f(B(0, \rho) ).\]
\item The following are equivalent:
\begin{enumerate}
\item $B(0,\rho)$ is the largest ball centered at $0$ where $f$ is injective;
\item there exists a point $q_0 \in \partial {B}(0,\rho)$ with $f(q_0) \in \partial{B}(0, \rho^2)$;
\item $f(q) = q \mr(q)$ where $\mr$ is a regular M\"obius transformation of $\B$ (necessarily such that $\mr(0)=\partial_cf(0)$, whence $|\mr(0)|=a$).
\end{enumerate}
\end{enumerate}
\end{theorem}

\begin{proof}
We prove each of the three properties separately.
\begin{enumerate}
\item
Since $\partial_cf(0)\neq 0$, by Proposition~\ref{realdifferential} we conclude that $f$ is a local diffeomorphism near $0$. Hence, there is a well-defined
\[r(f):= \sup\{r \in (0,1) : f \mathrm{\ is\ injective\ in\ } B(0,r)\}\]
and our thesis is $r(f)\geq \rho$.\\
Now, $f$ is a diffeomorphism from $B(0,r(f))$ onto its image while $f$ is not injective in $B(0,r)$ for any $r>r(f)$. By Theorem~\ref{globaltolocal}, there exists a point $q_0$ with $|q_0| = r(f)$, with $|f(q_0)| \leq (r(f))^2$. On the other hand, according to inequality~\eqref{eq:minmax}, 
\[|f(q_0)| \geq r(f) \frac{a-r(f)}{1-ar(f)}.\] The two inequalities together yield 
\[r(f) \geq \frac{a-r(f)}{1-ar(f)}.\]
The function $ r \mapsto \frac{a-r}{1-ar}$ from the real segment $(-1,1)$ to itself is strictly decreasing by Remark~\ref{quatmoebincrease} and it has a fixed point at $\rho$ by Lemma~\ref{extremalcase}. Therefore, the last inequality implies that $r(f)\geq \rho$, as desired.

\item
The right-hand inclusion
\[f({B}(0,r)) \subseteq {B}\left(0, r \frac{a+r}{1+ar}\right)\]
is an immediate consequence of inequality~\eqref{eq:minmax} and of the Maximum Modulus Principle~\ref{MaxPrinc}.\\
As for the left-hand inclusion, we reason as follows. If we take $r<\rho \leq r(f)$ and consider the preimage
\[U:=f^{-1}\Big(B\Big(0,r \frac{a-r}{1-ar}\Big)\Big)\]
then $U \subseteq B(0,r) \cup (\B \setminus \overline{B(0,\rho)})$. Indeed, if $q \in U$ had $r \leq |q| \leq \rho$ then, by inequality~\eqref{eq:minmax} and by Lemma~\ref{extremalcase},
\[|f(q)|\geq |q| \frac{a-|q|}{1-a|q|} \geq r \frac{a-r}{1-ar},\]
which would contradict the fact that $f(q) \in f(U) \subseteq B\big(0,r \frac{a-r}{1-ar}\big)$. Thus, the connected component $U_0$ of $U$ that includes the origin $0$ is such that
\[U_0 \subseteq B(0,r) \subset\subset B(0,r(f)).\]
In particular, $f$ is a diffeomorphism from the nonempty bounded open set $U_0$ onto its image 
\[V_0:=f(U_0)\subseteq B\Big(0,r \frac{a-r}{1-ar}\Big).\]
Moreover, $f$ maps the boundary $\partial U_0$ diffeomorphically to 
\[\partial V_0 \subseteq \overline{B\Big(0,r \frac{a-r}{1-ar}\Big)}.\] 
Let us prove, by contradiction, that $V_0=B\big(0,r \frac{a-r}{1-ar}\big)$. If this were not the case then $\partial V_0$ would include some interior point of the ball $B\big(0,r \frac{a-r}{1-ar}\big)$ and $\partial U_0$ would include some point of $U$. This is a contradiction. Indeed, by construction, $U_0$ is open and closed in the open set $U$, whence $U \cap \overline{U_0} = U \cap U_0$ does not intersect $\partial U_0$.\\
Therefore,
\[B\Big(0,r \frac{a-r}{1-ar}\Big) = V_0 \subseteq f(B(0,r)),\]
which proves the left-hand inclusion in property {\it 2}.\\
Finally, by taking the limit as $r \to \rho$, so that $r \frac{a-r}{1-ar} \to \rho^2$, we get 
\[{B}(0, \rho^2) \subseteq f(B(0, \rho) ),\]
as desired.

\item We prove three separate implications.\begin{enumerate}
\item[$(a) \Rightarrow (b)$]If $B(0,\rho)$ is the largest ball centered at $0$ where $f$ is injective then, by Theorem~\ref{globaltolocal}, there exists a point $q_0$ with $|q_0| =\rho$, with $|f(q_0)| \leq \rho^2$. On the other hand, according to inequality~\eqref{eq:minmax}, 
\[|f(q_0)| \geq \rho \frac{a-\rho}{1-a\rho} = \rho^2.\]
Therefore, there exists a point $q_0 \in \partial B(0,\rho)$ with $|f(q_0)| = \rho^2$.
\item[$(b) \Rightarrow (c)$]
If there exists a point $q_0 \in \partial B(0,\rho)$ with $f(q_0) \in \partial{B}(0, \rho^2)$ then an equality holds on the right-hand side of~\eqref{eq:minmax} at $q=q_0$. According to Theorem~\ref{minmax}, this implies that $f(q) = q \mr(q)$ where $\mr$ is a regular M\"obius transformation of $\B$ with $|\mr(0)|=a$.
\item[$(c) \Rightarrow (a)$]
If $f(q) = q \mr(q)$ where $\mr$ is a regular M\"obius transformation of $\B$ and if $a:=|\mr(0)|$ then, by Lemma~\ref{extremalcase}, $\partial_cf$ has a zero of modulus $\rho:=\frac{1-\sqrt{1-a^2}}{a}$. In such a case, by Theorem~\ref{singular} the function $f$ is not injective on any ball $B(0,R)$ with $R>\rho$.
\end{enumerate}
\end{enumerate}
\end{proof}

We conclude this section drawing from Theorem~\ref{landau} a useful consequence.

\begin{corollary}\label{landaubd}
Let $f : B(0,R) \to \hh$ be a bounded regular function. Set $C:= \sup_{q \in B(0,R)}|f(q)-f(0)|$, suppose that $a := \frac{R}{C} |\partial_cf(0)|$ belongs to $(0,1)$ and set $\rho := \frac{1-\sqrt{1-a^2}}{a}$. Then the following properties hold:
\begin{enumerate}
\item The function $f$ is injective at least in the ball $B(0, \rho R)$. 
\item For all $r \in (0, \rho)$, $B\big(f(0), r\frac{a-r}{1-ar} C\big) \subseteq  f({B}(0,rR)) \subseteq {B}\big(f(0), r \frac{a+r}{1+ar}C\big)$. As a consequence,
\[{B}\left(f(0), \rho^2C\right) \subseteq f\left(B(0, \rho R) \right).\]
\item The following are equivalent:
\begin{enumerate}
\item $B(0,\rho R)$ is the largest ball centered at $0$ where $f$ is injective;
\item there exists a point $q_0 \in \partial {B}(0,\rho R)$ with $|f(q_0)-f(0)|=\rho^2C$;
\item $f(q) = f(0) + \frac C R q \mr(\frac q R)$ where $\mr$ is a regular M\"obius transformation of $\B$ (necessarily such that $\mr(0)=\frac R C \partial_cf(0)$, whence $|\mr(0)|=a$).
\end{enumerate}
\end{enumerate}
\end{corollary}

\begin{proof}
It is an immediate consequence of Theorem \ref{landau}, applied to 
\[g(q):=\frac{1}{C}\left(f(qR)-f(0)\right),\]
since $g \colon \B \to \B$, $g(0)=0$ and $|\partial_c g (0) |=a.$
\end{proof}


\section{A new approach to the quaternionic Bloch--Landau theorem}\label{sec:bloch}

This section presents an application of Theorem~\ref{landau} to a new version of the quaternionic Bloch--Landau-type result proven in~\cite{blochlandau}. We will first prove a lemma and give a technical definition. In the statements and proofs, the symbol $B_I(0,R)$ denotes the disk $B(0,R)\cap L_I$ and $\B_I$ stands for $\B \cap L_I$, as usual.

\begin{lemma}\label{selfmap}
Let $g:B(0,R)\to\hh$ be a regular function. If, for some $I \in \s$, $|\partial_cg(x+yI)|$ is bounded by a constant $c>0$ in the disk $B_I(0,R)$ then $g(B(0,R)) \subseteq B(g(0),2cR)$.
\end{lemma}

\begin{proof}
Let $h(q):= g(q)-g(0)$. Our thesis will be proven if we show that $h(B(0,R)) \subseteq B(0,2cR)$.
For all $z \in B_I(0,R)$, let us denote by $\ell_z$ the line segment between $0$ and $z$. Then, as in the holomorphic case,
\[h(z) = \int_{\ell_z} \partial_c g (\zeta)\,d\zeta,\]
whence $|h(z)|\leq c|z| < cR$. By Proposition~6.10 of~\cite{weierstrass}, we conclude that $|h(q)| < 2cR$, as desired.
\end{proof}

\begin{definition}
Let $f:\B \to \hh$ be a regular function and let $p = x+Iy \in \B$. We define $f_p$ as the (unique) regular function on the ball $B(0,1-|p|)$ that coincides with $z \mapsto f(p+z)$ on the disk $B_I(0,1-|p|)$.
\end{definition}

\begin{remark}
When $p$ is in the real interval $(-1,1)$ then $f_p(q) = f(p+q)$ for all $q \in B(0,1-|p|)$.
\end{remark}

We now proceed with the announced new version of Theorem~6 of~\cite{blochlandau}.

\begin{theorem}\label{RegTras}
Let $\Omega\subseteq\hh$ be a symmetric slice domain that contains the closure of the unit ball $\B$ and let $f:\Omega \to \hh$ be a regular function with $|\partial_c f(0)|=1$. Then for all $I \in \s$ there exist $p \in \B_I$ and a ball $B$ centered at $0$ such that both of the following properties hold:
\begin{enumerate}
\item the function $f_p$ is injective in $B$;
\item the image $f_p(B)$ contains a ball $B(f(p),b)$ of radius $b>1/31$.
\end{enumerate}
As a consequence, there is a disk $D\subseteq\B_I$ centered at $p$ such that 
\begin{enumerate}
\item the function $f$ is injective in $D$;
\item the distance between $f(p)$ and $f(\partial D)$ is at least $b>1/31$.
\end{enumerate}
\end{theorem}

\begin{proof}
Let us fix $I \in \s$ and set
\[h(r):=(1-r) \max_{\partial B_I(0,r)} |\partial_c f|.\]
Then $h: [0,1] \to \rr$ is a continuous function. Since $h(0)=1$ and $h(1)=0$, there is a well-defined
\[r_0:=\max\{r \in [0,1] : h(r) = 1\}\in[0,1).\]
By the definition of $r_0$, $h(r)<1$ for all $r\in (r_0,1]$. Now let $p \in \partial B_I(0,r_0)$ be such that
\[|\partial_c f(p)| = \max_{\partial B_I(0,r_0)} |\partial_c f|= \frac{1}{1-r_0}.\]
If we set $\rho_0:=\frac{1-r_0}{2}$ and $\rho_1:=\frac{1+r_0}{2}$ then $|\partial_c f(p)| = \frac{1}{2\rho_0}$ and we have, for all $z \in B_I(0,\rho_1)\supset B_I(p,\rho_0)$,
\[|\partial_c f(z)| \leq \max_{\partial B_I(0,\rho_1)} |\partial_c f| = \frac{h(\rho_1)}{1-\rho_1} < \frac{1}{1-\rho_1} = \frac{1}{\rho_0}\]
where the first inequality follows from the Maximum Modulus Principle~\ref{MaxPrinc} and the second inequality follows from the fact that, by construction, $\rho_1>r_0$.

Let us consider the regular function $f_p: B(0,R) \to \hh$ with $R:=1-|p|$. It has $f_p(0) = f(p)$. Moreover, for all $z \in B_I(0,\rho_0)$,
\[\partial_c f_p(z) = \partial_c f(p+z),\]
whence $|\partial_c f_p(0)| =|\partial_c f(p)| = \frac{1}{2\rho_0}$ and $|\partial_c f_p(z)| \leq \frac{1}{\rho_0}$ for all $z \in B_I(0,\rho_0)$. By Lemma~\ref{selfmap}, $f_p(B(0,\rho_0))\subseteq B(f(p),2)$.

By applying Corollary~\ref{landaubd} to $f_p: B(0,\rho_0) \to B(f(p),2)$, we conclude that $f_p$ is injective in the ball $B(0,\rho_0 \rho)$ and that $f_p(B(0,\rho_0 \rho))$ includes $B\left(f(p),2\rho^2\right)$, where, after setting $a := \frac{\rho_0}{2} |\partial_c f_p(0)|  = \frac 14$,
\[\rho = \frac{1-\sqrt{1-a^2}}{a} = \frac{1-\sqrt{1-(1/4)^2}}{1/4} = 4-\sqrt{15}\]
and
\[2\rho^2 = 2(31- 8\sqrt{15})>1/31.\]
The main statement is thus proven, with $B:=B(0,\rho_0 \rho)$ and $b:=2\rho^2$.

The final statement follows from the definition of $f_p$, after setting $D:=B_I(p,\rho_0\rho)$ and observing that $f(\partial D)\subseteq f_p(\partial B)$.
\end{proof}

In the original result of~\cite{blochlandau}, the role of $B(f(p),b)$ was played by an open set of a different type. This slight improvement is a result of the new approach used here. The following problem is still open:

\begin{problem}
Find sufficient conditions on a regular $f:\overline{\B}\subset\Omega \to \hh$ to guarantee that the image of $f$ contains a ball (or another open set of a specific type) of universal radius. 
\end{problem}

The present work suggests that it might be possible to address the previous problem by first solving the following one:

\begin{problem}
Generalize Theorem~\ref{landau} and Corollary~\ref{landaubd}, studying the behavior at a point $p$ rather than at the origin.
\end{problem}

New work is envisioned to achieve the desired generalization, which will most likely not mimic the classical complex result but involve new exciting phenomena.


\section*{Acknowledgments}

This work was supported by the following grants of the Italian Ministry of Education (MIUR): PRIN {\it Variet\`a reali e complesse: geometria, topologia e analisi armonica}; Futuro in Ricerca {\it Differential Geometry and Geometric Function Theory}; Finanziamenti Premiali {\it SUNRISE}. It was also supported by the research group GNSAGA of INdAM.\\
We wish to thank the anonymous referee for carefully reading this article and for suggesting useful improvements to the presentation.

\end{document}